\documentclass{amsart}

\usepackage{amsmath}
\usepackage{amsthm,amssymb}
\usepackage{enumerate}
\usepackage[utf8]{inputenc}

\theoremstyle{plain}
\newtheorem{theorem}{Theorem}[section]
\newtheorem{lemma}[theorem]{Lemma}
\newtheorem{prop}[theorem]{Proposition}
\newtheorem{corollary}[theorem]{Corollary}
\theoremstyle{definition}

\newtheorem{remark}[theorem]{Remark}

\newcommand{\sign}{\mbox{\rm sign}}

\newcommand{\Dom}{\mbox{\rm DOM}}
\newcommand{\Char}{\mbox{\rm char}}
\newcommand{\Core}{\mbox{\rm Core}}

\newcommand{\N}{\mathbb N}
\newcommand{\Z}{\mathbb{Z}}

\newcommand{\R}{\mathbb{R}}

\begin{document}

	\title[Generalized algebraic rational identities]{Almost subnormal subgroups in division rings with generalized algebraic rational identities}
	\author[Bui Xuan Hai]{Bui Xuan Hai}
	\address{Faculty of Mathematics and Computer Science, VNUHCM-University of Science, 227 Nguyen Van Cu Str., Dist. 5, HCM-City, Vietnam}
	\email{bxhai@hcmus.edu.vn}
	\author[Truong Huu Dung]{Truong Huu Dung}
	\address{Faculty of Mathematics and Computer Science, VNUHCM-University of Science, 227 Nguyen Van Cu Str., Dist. 5, HCM-City, Vietnam}
	\email{thdung@dnpu.edu.vn}
	\author[Mai Hoang Bien]{Mai Hoang Bien}
	\address{Faculty of Mathematics and Computer Science, VNUHCM-University of Science, 227 Nguyen Van Cu Str., Dist. 5, HCM-City, Vietnam}
	\email{mhbien@hcmus.edu.vn}
	
	
		\maketitle
	\begin{abstract} In this paper we study non-central almost subnormal subgroups of the multiplicative group of a division ring satisfying a non-zero generalized rational identity. The main result generalizes Chiba's theorem on subnormal subgroups. As an application, we get a theorem on almost subnormal subgroups satisfying a generalized algebraic rational identity. The last theorem has several corollaries which generalize completely or partially some previous results.    
	\end{abstract}
	
	\section{Introduction, definitions and main results} 
	Let $D$ be a division ring with center $F$. We denote by $D^*$ the multiplicative group of $D$. In this paper, we study some algebraic valued functions in division rings. The notion of algebraic valued functions in division rings was firstly mentioned in \cite{her2} but some special cases were studied before. For instance, in \cite{kap}, Kaplansky proved that if $D$ satisfies a polynomial identity, then $D$ must be centrally finite, that is, $D$ is a finite dimensional vector space over $F$. This result later was extended to division rings with generalized polynomial identities \cite{Pa_Am_65}. In 1966, Amitsur studied rational identities in division rings \cite{Pa_Am_66} and he showed that if the center of a division ring $D$ is infinite and $D$ satisfies a rational identity, then $D$ is centrally finite. In 1970, Bergman \cite{Pa_Be_70} extended Amitsur's work and introduced the notion of generalized rational identities. In 1996, Chiba \cite{chiba} proved that if the center of a division ring $D$ is infinite and $D^*$ contains a non-central subnormal subgroup which satisfies a non-zero generalized rational identity, then $D$ is centrally finite. Our first aim in this paper is to carry over Chiba's result to almost subnormal subgroups. Namely, we shall prove the following theorem.
	\begin{theorem}\label{th1.1} Let $D$ be a division ring with infinite center. If $D^*$ contains a non-central almost subnormal subgroup which satisfies a generalized rational  identity $f=0$ for some non-zero $f(X)\in D(X)$, then $D$ is centrally finite.
	\end{theorem}
	Recall that in accordance with Hartley \cite{Pa_Ha_89}, a subgroup $N$ is {\it almost subnormal} in a group $G$ if there is a series of subgroups 
	$$N=N_r\le N_{r-1}\le\cdots\le N_0=G\eqno (1)$$
	of $G$ such that for each $1<i\le r$, either $N_{i}$ is normal in $N_{i-1}$ or $N_i$ has finite index in $N_{i-1}$. We call a series (1) an {\it almost normal series} in $G$. We say that a subgroup $N$ is {\it almost subnormal of defect} (or {\it depth}) $r$ if there is a series (1) of length $r+1$ and if no such a series of  lesser length exists. The existence of almost subnormal subgroups that are not subnormal in division rings have been noted in \cite{Pa_NgBiHa_2016}. 
	
	Now, let us recall the definition of generalized rational identities. Let $D$ be a division ring with center $F$ and $X=\{x_1,\cdots,x_m\}$ be $m$ (noncommuting) indeterminates. We denote by $F\langle X\rangle$ the free $F$-algebra on $X$, by $D\langle X\rangle$ the free product of $D$ and $F\langle X\rangle$ over $F$, and by $D(X)$ the universal division ring of fractions of $D\langle X\rangle$. The existence of $D(X)$ was shown and studied deeply in \cite[Chapter 7]{Bo_Cohn}. One calls an element $f(X)\in D(X)$ a \textit{generalized rational polynomial}.  If $f(X)\in D(X)$, then by \cite[Theorem 7.1.2]{Bo_Cohn}, $f(X)$ is an entry of the matrix $A^{-1}$,  where $A\in M_n(D\langle X\rangle)$ for some positive integer $n$ such that $A$ is invertible in $M_n(D(X))$. If $f(X)\in F\langle X\rangle$, then we say that $f(X)$ is a \textit{polynomial} in $X$ over $F$. If $f(X)\in D\langle X\rangle$, then $f(X)$ is a \textit{generalized polynomial} in $X$ over $D$. If all coefficients of a generalized rational polynomial $f(X)$ are in $F$, then $f(X)$ is said to be a \textit{rational polynomial} in $X$. Let $c=(c_1,\cdots,c_m)\in D^m$ and $\alpha_c\colon D\langle X\rangle \to D$ be the ring homomorphism defined by $\alpha(x_i)=c_i$. For any $n\in \N, $ let $S(c,n)$ be the set of all square matrices $(f_{ij}(X))$ of degree $n$ over $D\langle X\rangle$ such that the matrix $(f_{ij}(c))$ is invertible in $M_n(D)$. Let $S(c)=\bigcup\limits_{n\ge 1}S(c,n)$ and $E(c)$ be the subset of $D(X)$ consisting of all entries of $A^{-1}$, where $A$ ranges over $S(c)$. Then, $E(c)$ is a subring of $D(X)$ containing $D\langle X\rangle$ as a subring. Moreover, there is a ring homomorphism $\beta_c\colon E(c)\to D$ which extends $\alpha_c$ and every element of $E(c)$ is invertible if and only if the matrix mapped by $\beta_c$ is nonzero. Let $f(X)\in D(X)$ and $c\in D^m$. If $f(X)\in E(c)$, then we say that $f(X)$ is defined at $c$ and $\beta_c(c)$ is denoted by $f(c)$. For any $f(X)\in D(X)$, the set of all $c\in D^m$ such that $f(X)$ is defined at $c$ is called the {\it domain} of $f(X)$ and is denoted by $\Dom_D(f)$. Let $f(X)\in D(X)$ and $S\subseteq D$. If $f(c)=0$ for all $c\in S^m \cap \Dom_D(f)$, then we say that $f=0$ is a {\it generalized rational identity}  (briefly, GRI) of $S$ or $S$ {\it satisfies the generalized rational identity} $f=0$. We refer to \cite[Chapter 7]{Bo_Cohn} and \cite[Chapter 8]{Bo_Rowen} for further reading on generalized rational identities.
	
	Further, we use Theorem \ref{th1.1} to study some algebraic valued functions in division rings. In \cite[Theorem 4]{kap}, Kaplansky proved that a division ring whose elements all are algebraic over its center of bounded degree is centrally finite.  Recently, Bell et {\it al.} considered the left algebraicity over a subfield of a division ring and they proved the analogue theorem for division rings with this property \cite{bell}. In \cite[Theorem 4]{ak}, it was proved that if in a division ring $D$ with center $F$ all additive commutators $xy-yx$ are algebraic over $F$ of bounded degree, then $D$ is centrally finite. An analogue result for multiplicative commutators was also obtained in \cite{chebo}. Namely, $D$ is centrally finite in the case when either $xyx^{-1}y^{-1}$ are algebraic over $F$ of bounded degree for all $x,y\in D^*$ or $\Char (D)=0$ and there exists a non-central element $a$ such that $axa^{-1}x^{-1}$ are  algebraic over $F$ of bounded degree for all $x\in D$ (see \cite[theorems 6, 7]{chebo}). Moreover, in \cite[Theorem 4]{ak} and \cite[theorems 6, 7]{chebo}, the bound of $[D:F]$ was also given. In the following, we give the definition of a \textit{generalized algebraic rational identity} (GARI) which generalizes the notion of generalized rational identity GRI given above. As an application of Theorem~\ref{th1.1}, we get some theorems which generalize completely or partially results mentioned above. 
	
	Recall that an element $\alpha\in D$ is called {\it algebraic} over $F$ if $\alpha$ is a root of non-zero polynomial over $F$, that is $$a_n\alpha^n+a_{n-1}\alpha^{n-1}+\cdots+a_1\alpha+a_0=0$$ for some $n\in \N$, $a_i\in F$ and $a_n\ne 0$. Let $S$ be a subset of $D$ and  $f(X)$ be a generalized rational polynomial over $D$. We say that $f$ is a {\it generalized algebraic rational identity} (briefly, GARI) of $S$ (or  $S$ \textit{satisfies the} GARI $f$) if $f(c)$ is algebraic over $F$ whenever $c=(c_1,\cdots, c_m)\in S^m\cap \Dom_D(f)$. A GARI $f$ of $S$ is \textit{non-trivial} if there exists a division ring $D_1$ containing $D$ as a division subring such that the center $Z(D_1)$ of $D_1$ is exactly $F$ and there exists an element $a=(a_1,\cdots,a_m)\in \Dom_{D_1}(f)$  such that $f(a)$ is not algebraic over $F$. The class of non-trivial GARIs is very big. In the next, we will show that every non-zero GRI is a non-trivial GARI. Using Theorem \ref{th1.1}, we shall prove the following theorem.
	\begin{theorem}\label{th1.2} Let $D$ be a division ring with infinite center and assume that $N$ is a non-central almost subnormal subgroup of $D^*$. If $N$ satisfies a non-trivial {\rm GARI} of bounded degree, then $D$ is centrally finite.
	\end{theorem}
	
	Theorem \ref{th1.2} has the following corollaries.
	
	\begin{corollary}\label{cor1.3} Let $D$ be a division ring with center $F$ and assume that $N$ is a non-central almost subnormal subgroup of $D^*$. If all elements of $N$ are algebraic over $F$ of bounded degree, then $D$ is centrally finite. 
	\end{corollary}
	\begin{corollary}\label{cor1.4} Let $D$ be a division ring with infinite center $F$ and $N$ is a non-central almost subnormal subgroup of $D^*$. If  $xy-yx$ is algebraic over $F$ of bounded degree for all $x, y\in N$, then $D$ is centrally finite. 
	\end{corollary}
	\begin{corollary}\label{cor1.5} Let $D$ be a division ring with center $F$ and assume that $N$ is a non-central almost subnormal subgroup of $D^*$. If $\Char(D)=0$ and there exists an element $a\not\in F$ such that $axa^{-1}x^{-1}$ is algebraic over $F$ of bounded degree for all $x\in N$, then $D$ is centrally finite.
	\end{corollary}
	\begin{corollary}\label{cor1.6} Let $D$ be a division ring with center $F$ and assume that $N$ is a non-abelian almost subnormal subgroup of $D^*$. If $xyx^{-1}y^{-1}$ is algebraic over $F$ of bounded degree $d$ for all $x,y\in N$, then $D$ is centrally finite.
	\end{corollary}
	Corollary \ref{cor1.3} is a broad extension of \cite[Corollary 3]{mah} and the Jacobson Theorem \cite[Theorem 7]{jab}; Corollary \ref{cor1.4}  partially extends \cite[Theorem 4]{ak}; Corollary \ref{cor1.5} and Corollary \ref{cor1.6} partially extend \cite[Theorem 7]{chebo} and \cite[Theorem 6]{chebo} respectively. It is unfortunate that in corollaries \ref{cor1.4}, \ref{cor1.5} and \ref{cor1.6} we are not able to give the bound for the dimension of $D$ over $F$.    

\section{Proof of Theorem \ref{th1.1}}
Let $D$ be a division ring with center $F$ and $\phi$ be a ring automorphism of $D$. We write $D((\lambda,\phi))$ for the ring of skew Laurent series $\sum\limits_{i = n}^\infty  {{a_i}{\lambda^i}}$, where $n\in \Z, a_i\in D$, with the multiplication defined by the twist equation $\lambda a=\phi(a)\lambda$ for every $a\in D$. If $\phi=Id_D$, then we write $D((\lambda))$ instead of $D((\lambda,Id_D))$. It is known that  $D((\lambda,\phi))$ is a division ring (see \cite[Example 1.8]{lam}). Moreover, we have the following results.

\begin{lemma}\label{lem2.1}{\rm \cite[Lemma 2.1]{bien_13}} Let $D$ be a division ring with center $F$. Assume that $K=\{\, a\in D\mid \phi(a)=a\}$ is the fixed division subring of $\phi$ in $D$. If the center $k=Z(K)$ of $K$ is contained in $F$, then the center of $D((\lambda,\phi))$ is 
	
	$$Z(D((\lambda,\phi)))=\left\{ {\begin{array}{*{20}{c}}
		k&{\text{ if } \phi \text{ has infinite order, }}\\
		{k(({\lambda^s}))}&{\text{ if } \phi \text{ has a finite order } s.}
		\end{array}} \right.$$
		In particular, the center of $D((\lambda))$ is $F((\lambda))$.  
		\end{lemma}

\begin{lemma}\label{lem2.2}
	Let $D$ be a division ring and $f(X)\in D(X)$. Then $f(X)\ne 0$  if and only if there exist a division ring $D_1$ containing $D$ as a division subring and an $m$-tuple $a=(a_1,\cdots,a_m)\in \Dom_{D_1}(f)$ such that $f(a)\ne 0$.
\end{lemma}

\begin{proof}
	($\Rightarrow$) Assume that $f(X)\ne 0$. Let $F$ be the center of $D$. By Lemma~\ref{lem2.1}, the center $F((\lambda))$ of $D((\lambda))$ is infinite. Since $f(X)$ is also a generalized rational polynomial over $D((\lambda))$, without loss of generality, we can assume that $F$ is infinite because if it is necessary, we can consider $D((\lambda))$ instead of $D$. 
	Let $\lambda_1,\lambda_2$ be noncommuting indeterminates. Denote by $D_1=D(\lambda_1,\lambda_2)$ the universal division ring of fractions of $D\langle \lambda_1,\lambda_2\rangle$. Then, the center of $D_1$ is $F$ and $D_1$ is centrally infinite. It is clear that $D_1$ contains $D$ as a division subring and $f(X)$ is a non-zero generalized rational polynomial over $D_1$. Now, in view of \cite[Theorem 1]{chiba}, there exists $a=(a_1,\cdots,a_m)\in \Dom_{D_1}(f)$ such that $f(a)\ne 0$.
	
	($\Leftarrow$) This implication is trivial.
\end{proof}

\begin{lemma}\label{lem2.3}
	Let $D$ be a division ring and $f(X)\in D(X)$. Assume that $f=0$ is a non-zero GRI of a subgroup $S$ of $D^*$. If $(c_1,\cdots,c_m)\in S^m\cap \Dom_D(f)$ and if we put $$g(y_1,\cdots,y_m)=f(c_1y_1,\cdots,c_my_m),$$ then $g=0$ is a non-zero GRI of $S$. As a corollary, if $f$ is a non-zero {\rm GRI} of $S$, then either $S^m\cap \Dom_D(f)= \emptyset$ or $S$ satisfies  some non-zero {\rm GRI} $g$ such that $(1,\cdots,1)\in S\cap \Dom_D(g)$.
\end{lemma}
\begin{proof} Assume that $(c_1,\cdots,c_m)\in S^m$ such that $f$ is defined at $(c_1,\cdots,c_m)$. Put $$g(y_1,\cdots,y_m)=f(c_1y_1,\cdots,c_my_m).$$ It is obvious that $g=0$ is a GRI of $S$.  
By Lemma~\ref{lem2.2}, to show $g\ne 0$, it suffices to find a division ring $D_1$ containing $D$ and $a=(a_1,\cdots,a_m)\in D_1^m$ such that $g$ is defined at $a$ and $g(a)\ne 0$. Indeed, since $f(X)\ne 0$, by Lemma~\ref{lem2.2}, there exist a division ring $D_1$ containing $D$ and $a=(a_1,\cdots,a_m)\in D_1^m$ such that $f$ is defined at $a$ and $f(a)\ne 0$. Since $c_i$ is invertible, $g(x_1,\cdots, x_m)$ is defined at $b$ and $g(b)\ne 0$ with $b=(c_1^{-1}a_1, \cdots, c_m^{-1}a_m)$. The proof of the lemma is complete.	\end{proof}

The next two lemmas are due to Chiba.

\begin{lemma}\label{lem2.4} {\rm \cite[Lemma 5]{chiba}}
	Let $h$ be a non-central element and let $f(X)$ be a non-zero element of $D(X)$ defined at $(1,\cdots, 1)$. Then, for $2m$ indeterminates $y_1,z_1,y_2,z_2,\cdots,y_m,z_m$, we have $$g(y_1,\cdots,y_m,z_1,\cdots,z_m)=f([z_1,y_1hy_1^{-1}],\cdots,[z_m,y_mhy_m^{-1}])$$$$=f(z_1y_1hy_1^{-1}z_1^{-1}y_1h^{-1}y_1^{-1},\cdots, z_my_mhy_m^{-1}z_m^{-1}y_mh^{-1}y_m^{-1})\in D(X)^*.$$  
\end{lemma}
\begin{remark}\label{remark2.5}
	Assume that $N$ is a non-central almost subnormal subgroup of $D^*$ with an almost normal series
	$$N=N_r\le N_{r-1}\le\cdots\le N_0=D^*,$$
	and $N$ satisfies a non-zero GRI.   
	If $N_r$ has finite index in $N_{r-1}$, then 
	$$\Core_{N_{r-1}}(N_r) = \bigcap\limits_{x\in N_{r-1}}xN_rx^{-1}$$
	is the normal subgroup of finite index in $N_{r-1}$. We claim that $\Core_{N_{r-1}}(N_r)$ is non-central. Indeed, if $\Core_{N_{r-1}}(N_r)$ is central, then $a^n\in \Core_{N_{r-1}}(N_r)$ for any $a\in N_{r-1}$, where $n=[N_{r-1}:\Core_{N_{r-1}}(N_r)]$. Hence, $N_{r-1}$ satisfies the identity $x^ny^nx^{-n}y^{-n}=1$. By \cite[Theorem 2.2]{Pa_NgBiHa_2016}, $N_{r-1}$ is central, so is $N$, a contradiction. Thus, $\Core_{N_{r-1}}(N_r)$ is a non-central almost subnormal subgroup of $D^*$ and $N$ satisfies a non-zero \textit{GRI}. 
\end{remark}
\begin{lemma}\label{lem2.6} {\rm \cite[Lemma 6]{chiba}}
	Let $D$ be a centrally infinite division ring with infinite center and let $N$ be a non-central almost subnormal subgroup of $D^*$. For any element $f(X)$ of $D(X)$, if $f(X)\ne 0$, then $N^m\cap \Dom_D(f)\ne \emptyset$.
\end{lemma}
\begin{proof}  Assume that 
$$N=N_r\le N_{r-1}\le\cdots\le N_1\le N_0=D^*$$  
is an almost normal series of $N$ in $D^*$. In view of Remark~\ref{remark2.5}, replacing $N_r$ by $\Core_{N_{r-1}}(N_r)$ if it is necessary, we can assume that $N_r\trianglelefteq N_{r-1}$. One can see that the conclusion of the lemma is obvious by applying the same  arguments as in the proof  of \cite[Lemma 6]{chiba}.
\end{proof}

Now we are ready to prove Theorem \ref{th1.1}.

\begin{proof}  Assume that $N$ is an almost subnormal subgroup  of $D^*$ of defect $r$ with an almost normal series 
$$N=N_r\le N_{r-1}\le\cdots\le N_1 \le N_0=D^*,$$ 
and $N$ satisfies a generalized rational identity $f=0$, where $f(X)\in D(X)^*$.  We shall prove the statement by induction on $r$. 
If $r=0$, that is, $D^*$ satisfies a non-zero GRI, then by \cite[Theorem 1]{chiba}, $D$ is centrally finite. Assume that the statement is true for any non-central almost subnormal subgroup $N$ of defect less than $r$. We must show that the statement holds for all non-central almost subnormal subgroups of $D^*$ of defect $r$. 
	Assume that $N$ is a non-central almost subnormal subgroup of $D^*$ of defect $r$ such that $f=0$ is a GRI of $N$. 
	By Remark~\ref{remark2.5}, let 
	$$N=N_r\le N_{r-1}\le\cdots\le N_1 \le N_0=D^*$$ 
	be an almost normal series of $N$ in $D^*$ of defect $r$ with $N_r\trianglelefteq N_{r-1}$. Deny the statement, assume that $D$ is centrally infinite.   
	By Lemma~\ref{lem2.6}, $f$ is defined at some $(c_1,\cdots,c_m)\in N^m$.  In view of Lemma~\ref{lem2.3}, we can suppose that $f$ is defined at $(1,\cdots, 1)$. Since $N_r$ is non-central, there exists $h\in N_r\backslash F$. By Lemma~\ref{lem2.4}, we have  $$g(y_1,\cdots,y_m,z_1,\cdots,z_m)=f([z_1,y_1hy_1^{-1}],\cdots,[z_m,y_mhy_m^{-1}])$$$$=f(z_1y_1hy_1^{-1}z_1^{-1}y_1h^{-1}y_1^{-1},\cdots, z_my_mhy_m^{-1}z_m^{-1}y_mh^{-1}y_m^{-1})\ne 0.$$ Observe that $[a,bhb^{-1}]\in N_{r}$ for any $a,b\in N_{r-1}$, so $g=0$ is also a non-zero GRI of $N_{r-1}$. By the inductive hypothesis, $D$ is centrally finite, a contradiction. 
\end{proof}

\section{GRI and GARI}

We will see that the class of non-trivial GARIs is very large. In fact, we will show in this section that the class of non-trivial GARIs contains non-zero GRIs and non-trivial generalized power central rational identities (see below for the definition) in some cases. For a given positive integer $n$, let  $x,y_1,\ldots, y_n$ be  $n+1$ noncommuting indeterminates.  Consider the following generalized rational polynomial 
$$g_n(x,y_1,\ldots, y_n)=\sum\limits_{\sigma  \in {S_{n + 1}}} {\sign(\sigma ){x^{\sigma (0)}}{y_1}{x^{\sigma (1)}} \ldots {y_n}{x^{\sigma (n)}}}, $$ where $S_{n+1}$ is the symmetric group defined on the set  $\{\,0,1,\ldots, n\,\}$. 
\begin{lemma}\label{3.1} Let $D$ be a division ring with center $F$. For any element $a\in D$, the following statements are equivalent:
	\begin{enumerate}
		\item The element $a$ is algebraic over $F$ of degree $\le n$.
		\item $g_n(a,y_1,\ldots, y_n)=0$ is a {\rm GRI} on $D$.
	\end{enumerate} 
\end{lemma}
\begin{proof} See \cite[Corollary 2.3.8]{beidar}.
\end{proof}

\begin{lemma}\label {3.2} Let $D$ be a division ring with center $F$. An element $\alpha=a_1\lambda+a_2\lambda^2+\cdots$ in $D((\lambda))$ is algebraic over $F$ if and only if $\alpha=0$.
\end{lemma}
\begin{proof} Suppose that $\alpha \ne 0$ and  $g(x) =t_0 +t_1x + \cdots + t_n x^n \in F[x]$ is the minimal polynomial of $\alpha$ over $F$.  Then, from the equality $$0=t_0 + t_1(a_1\lambda +a_2 \lambda^2 + \cdots) + \cdots + t_n (a_1\lambda +a_2 \lambda^2 + \cdots)^n$$ one has $t_0=0$, that is impossible. 
\end{proof}

For a division ring $D$ with  center $F$, let us consider a countable set of indeterminates $\{\,\lambda _i\mid i\in \Z\,\}$ and  a family of division rings which is constructed by setting 
$$D_0=D((\lambda _0)), D_1 =D_0((\lambda _1)),$$  $$D_{-1}=D_1((\lambda _{-1})), D_2=D_{-1}((\lambda _{2})),$$ 
for any $n>1,$ $$ D_{-n}=D_n((\lambda _{-n})),D_{n+1}=D_{-n}((\lambda _{n+1})).$$  
Clearly, $D_{\infty}=\bigcup\limits_{n=-\infty}^{+\infty} {{D_n}}$ is a division ring. By Lemma~\ref{lem2.1}, it is not hard to prove by induction on $n\ge 0$ that the center of $D_0$ is $F_0=F((\lambda _0))$, the center of $D_{n+1}$ is $F_{n+1}=F_{-n}((\lambda _{n+1}))$ and the center of $D_{-n}$ is $F_{-(n+1)}=F_{n+1}((\lambda _{-(n+1)}))$.  In particular, $F$  is contained in $Z(D_\infty)$. Consider the map $\phi: D_\infty\longrightarrow D_\infty$ which is defined by $\phi(a)=a$ for any $a\in D$ and $\phi(\lambda _i)=\lambda _{i+1}$ for any $i\in \Z$ is an automorphism of $D_\infty$. 

\begin{prop}\label{pro3.3} The center of $D_\infty((\lambda ,\phi))$ is $F$.  
\end{prop}
\begin{proof} We note that  $D$ is the fixed division ring of $\phi$ in $D_\infty$. Since  $F$  is contained in the center of $D_\infty$, the automorphism $\phi$ has infinite order. By Lemma~\ref{lem2.1}, $Z(D_\infty((\lambda ,\phi)))=F.$\end{proof} 

\begin{theorem}\label{th3.4} Let $D$ be a division ring with center $F$ and $S$ be a subset of $D$. Assume that $f(X)\in D(X)\backslash D$ is a {\rm GARI} of $S$. If $f(c)\in F$ for some $c=(c_1,\cdots,c_m)\in S^m$, then $f$ is a non-trivial {\rm GARI}.
\end{theorem}

\begin{proof} By the definition of non-trivial GARIs, we must find a division ring $L$ with center $F$ containing $D$ as a division subring and an element $a=(a_1,\cdots,a_m)$ in $L^m$ such that  $f(a)$ is not algebraic over $F$. Let $K=D(Y)$, where $Y=\{y_1,\cdots,y_m\}$ and  $L=K_{\infty}((\lambda,\phi))$.  By Proposition~\ref{pro3.3}, $Z(L)=Z(K)=Z(D)=F$. Consider the division subring $K((\lambda))$ of $L$. By  Lemma~\ref{lem2.1}, $F((\lambda))$ is the center of $K((\lambda))$. In view of \cite[Lemma 7]{chiba_88}, 
	$$f(c+\lambda Y)=f(c_1+y_1\lambda, \cdots, c_m+y_m\lambda)=f(c)+\sum_{j=1}^\infty f_j(Y)\lambda^j,$$
	where $f_j(Y)$ are generalized polynomials over $D$ and there is $j_0$  such that $f_{j_0}(Y) \ne 0$. Since $f(c)\in F$, if 
	$f(c+\lambda Y)$
	is algebraic over $F$, then $\sum_{j=1}^\infty f_j(Y)\lambda^j$ is algebraic over $F$ too. By Lemma~\ref{3.2}, $f_j(Y)= 0$ for every $j\ge 1$. In particular, we have 
	$f_{j_0}(Y)= 0,$
	a contradiction. Thus, $f(c+\lambda Y)$ is not algebraic over $F$. Therefore, $f$ is not a GARI of $L$.
\end{proof}
Recall that for a division ring $D$ with center $F$, an element $f(X)\in D(X)$ is called a {\it generalized power central rational identity} (shortly, GPCRI) of a subset $S$ of $D$ if $f(X)$ satisfies the following condition: if $c=(c_1,\cdots,c_m)\in S^m\cap \Dom_D(f)$, then there exists a positive integer $p_c$ (depending on $c$) such that $f(c)^{p_c}\in F$ (see \cite{chiba_88}). Moreover, if $f(X)^p\in D(X)\backslash F$ for any positive integer $p$, then we say that $S$ satisfies a \textit{non-trivial}  GPCRI $f(X)$. It is obvious that if $f(X)$ is a GPCRI of some subset $S$, then $f(X)$ is a GARI of $S$. In the following result, we will show that every non-trivial GPCRI is a non-trivial GARI. 
\begin{corollary}\label{cor3.5} Let $D$ be a division ring with center $F$ and  $f(X)\in D(X)$. Assume that $S$ is a subset of $D$ such that $S^m\cap \Dom_D(f)\ne \emptyset$. If $f(X)$ is a non-trivial GPCRI of $S$, then $f(X)$ is a non-trivial {\rm GARI} of $S$.
\end{corollary}
\begin{proof} Assume that $f(X)$ is a non-trivial GPCRI of $S$. Then, it is obvious that $f(X)$ is a GARI of $S$. Now we will show that $f(X)$ is a non-trivial GARI of $S$. It suffices to prove that $g(X)=f(X)^p$ is a non-trivial GARI of $S$ for some positive integer $p$. Indeed, since $S^m\cap \Dom_D(f)\ne \emptyset$, let  $c=(c_1,\cdots,c_m)\in S^m\cap \Dom_D(f)$. Then, there exists $p_c>0$ such that $f(c)^{p_c}\in F$. Put $g(X)=f(X)^{p_c}$. Then it is obvious that $g(X)$ is a GARI of $S$. Since $g(c)=f(c)^{p_c}\in F$ which implies that $g(X)$ is a non-trivial GARI of $S$ by Theorem~\ref{th3.4}.
\end{proof}

Let $D$ be a division ring with  center $F$ and $S$ be a  subset of $D^*$. We say that $S$ {\it satisfies a non-trivial {\rm GARI} of bounded degree} if there exists  a non-trivial GARI $f(X)$ of $S$ over $D$ such that for all $c=(c_1, \ldots, c_m)\in S^m\cap \Dom_D(f)$, the element $f(c)$ are algebraic over $F$ of bounded degree.

Notice that if $f=0$ is a GRI of a subset $S$ of a division ring $D$, then obviously $f(X)$ is a GARI of $S$. However, if in addition, $f(X)$ is non-zero, then  $f(X)$ may not be a non-trivial GARI. For example, let $$D=\R \oplus \R i \oplus \R j \oplus \R k$$ be a real quaternion algebra,  $S=\{1\}$ and $f(X)=(X-1)i(X-1)^{-1}$. Then $S^1\cap \Dom_D(f)=\emptyset$ which implies that $f=0$ is a non-zero GRI of $S$. Moreover, since $f(X)^2=(X-1)i^2(X-1)^{-1}=-1$, $f(X)$ is a GARI (of degree $2$) of any division ring $D_1$ containing $D$. Hence, $f(X)$ is a trivial GARI of $S$. The following result shows that if $f(X)$ is defined at least at an $m$-tuple in $S^m$, then $f(X)$ is a non-trivial GARI of $S$.

\begin{corollary}\label{cor3.6} Let $D$ be a division ring with center $F$ and $f(X)\in D(X)$. Assume that $S$ is a subset of $D$ such that $S^m\cap \Dom_D(f)\ne \emptyset$. If $f=0$ is a non-zero {\rm GRI} of $S$, then $f(X)$ is a non-trivial {\rm GARI} of $S$ of degree $1$.
\end{corollary}
\begin{proof} Assume that $f=0$ is a non-zero GRI of $S$. Then $f(X)$ is a GARI of $S$ of degree $1$. Now we must show $f(X)$ is non-trivial GARI of $S$. Indeed, if $f(X)\in D$, then $f(X)=0$ that is impossible since $f=0$ is a non-zero GRI of $S$. Hence $f(X)\not\in D$. By the hypothesis, there exists $c=(c_1, \ldots, c_m)\in S^m\cap \Dom_D(f)$ such that $f(c)=0$. In view of Theorem \ref{th3.4}, it follows that $f(X)$ is a non-trivial GARI of $S$.
\end{proof}

\noindent{\bf Remark.}  In view of Corollary \ref{cor3.6}, one can see that Theorem~\ref{th1.2} generalizes \cite[Theorem 1]{chiba}. Note that the last theorem is a generalization of a theorem of Makar-Limanov  \cite[Theorem]{Pa_Ma_88}.

\section{Proof of Theorem \ref{th1.2}}

\begin{proof} Assume that $N$ satisfies the non-trivial GARI $f(X)$ of bounded degree $k$. Let $Y=\{y_1,\cdots, y_k\}$ and $x$ be $k+1$ indeterminates. Consider 
	$$g_k(x,Y)=\sum\limits_{\sigma  \in {S_{k + 1}}} {\sign(\sigma ).{x^{\sigma (0)}}{y_1}{x^{\sigma (1)}} \ldots {y_k}{x^{\sigma (k)}}}$$ as in Lemma \ref{3.1},  and put $$w(X, Y)=g_k(f(X),Y).$$ 
	Assume that $c=(c_1,c_2,\cdots,c_m)\in N^m\cap \Dom_D(f)$. Since $f(c)$ is algebraic over $F$ of degree $\le k$, by Lemma~\ref{3.1}, $$w(c, r)=0$$ for every $r=(r_1,\cdots,r_k)\in D^k$. In particular, $w=0$ is a GRI of $N$. Because $f$ is a non-trivial GARI,  there exist a division ring $D_1$ with center $F$ containing $D$ as a division subring and an element $a=(a_1,\cdots, a_m)\in \Dom_{D_1}(f)$ such that $f(a)$ is not algebraic over $F$. By Lemma~\ref{3.1}, $w(X,Y)\ne 0$, and consequently, $w=0$ is a non-zero GRI of $N$. Now in view of Theorem~\ref{th1.1}, $D$ is centrally finite, and the proof of Theorem \ref{th1.2} is now complete.
\end{proof}

\noindent{\bf Remark.} The  non-triviality of $f$  in Theorem~\ref{th1.2} is essential. For instance, let $D$ be a centrally infinite division ring,  and assume that $a\in D^*$ is an algebraic element of degree $d$ over the center $F$ of $D$. Then, for any $b\in D^*$, the element $bab^{-1}$ is always algebraic over $F$ of degree $\le d$. This means that $f(x)=xax^{-1}$ is a GARI of $D^*$ while $D$ is centrally infinite. %$[D: F]=\infty$. 
\bigskip

\section{Proofs of corollaries}
\subsection{Proof of Corollary \ref{cor1.3}}
\begin{proof}If the center $F$ is finite, then every element of $N$ is torsion. By \cite[Proposition 4.4]{Pa_NgBiHa_2016} $N$ is central which contradicts the hypothesis. Hence, $F$ is infinite. By Theorem~\ref{th1.2}, $D$ is centrally finite.\end{proof}
\subsection{Proof of Corollary \ref{cor1.4}}

\begin{proof} The corollary is followed directly from Theorem~\ref{th1.2}. \end{proof}

\subsection{Proof of Corollary \ref{cor1.5}}
\begin{proof} Put $w(x)=axa^{-1}x^{-1}$. Then $w(x)$ is a GARI of $N$ of bounded degree. Using Theorem~\ref{th3.4}, $w(x)$ is a non-trivial GARI of $N$ because $w(1)=1\in F$. Since $F$ is infinite, by Theorem~\ref{th1.2}, $D$ is centrally finite.
\end{proof}
\subsection{Proof of Corollary \ref{cor1.6}}

\begin{proof} Put $w(x,y)=xyx^{-1}y^{-1}$. Then $w(x)$ is a GARI of $N$ of bounded degree $d$. In view of Theorem~\ref{th3.4}, $w$ is a non-trivial GARI of $N$ because $w(1,1)=1\in F$. By Theorem~\ref{th1.2}, it suffices to show that $F$ is infinite. Indeed, assume that $F$ is finite. Then, for any $a,b\in N$, the subfield $F(aba^{-1}b^{-1})$ of $D$ generated by $aba^{-1}b^{-1}$ over $F$ is finite which implies that $aba^{-1}b^{-1}$ is torsion of order $n\le |F|^d-1$. Therefore, $N$ satisfies a generalized group identity $w(x,y)^n=1$. If there exist $a,b\in N$ such that $aba^{-1}b^{-1}=w(a,b)\not\in F$, then there exists a division subring $D_1$ of $D$ with center $F_1$ such that $D_1$ is centrally finite and $aba^{-1}b^{-1} \in D_1\backslash F_1$ (\cite[Proposition 2.1]{bd}). Since $N$ is almost subnormal in $D^*$, there exist a series of subgroups $$N=N_r \le N_{r-1}\le\cdots \le N_1=D^*$$ 
such that for any $1\le i<r$, either $[N_i:N_{i+1}]<\infty$ or $N_{i+1}$ is normal in $N_i$. Putting $H_i=D_1\cap N_i$,  we obtain the following series of subgroups 
$$H_r \le H_{r-1}\le\cdots \le H_1=D_1^*,$$ 
where for any $1\le i<r$, either $[H_i:H_{i+1}]<\infty$ or $H_{i+1}$ is normal in $H_i$. Therefore $H=H_r$ is an almost subnormal subgroup of $D_1^*$ and $aba^{-1}b^{-1}\in H$. Observe that $F_1$ is infinite and since $H\subseteq N$ satisfies the group identity $w(x,y)^n=1$, by \cite[Theorem 2.2]{Pa_NgBiHa_2016}, $H\subseteq F_1$. In particular, $aba^{-1}b^{-1}\in F_1$, a contradiction. Hence, $aba^{-1}b^{-1}\in F$ for any $a,b\in N$. Since $N$ is non-abelian, there exist $a,b\in N$ such that $aba^{-1}b^{-1}=\alpha\in F\backslash \{1,0\}$. Let $n$ be the order of $\alpha$. Consider the division subring $F(a,b)$ of $D$ generated by $a,b$ over $F$. Then, we also have $ab^na^{-1}=(\alpha b)^n=b^n$ because $aba^{-1}=\alpha b$, hence $b^n\in Z(F(a,b))$. Similarly, we also have $a^n\in Z(F(a,b))$. If $a^n$ and $b^n$ are algebraic over $F$, then of course $a$ and $b$ are algebraic over $F$ too. In this case, it is easy to check that $F(a,b)=F[a,b]$ is finite, hence $F(a,b)$ is commutative that is impossible because $ab\ne ba$. Therefore, $a^n$ or $b^n$ is not algebraic over $F$. As a result, $Z(F(a,b))$ is infinite.
	By repeating the arguments in the first part with $F(a,b)$ instead of $D_1$, we have $F(a,b)\cap N$ is abelian. Again, this is impossible in view of inequality $ab\ne ba$. Thus, $F$ is infinite and the proof of the corollary is now complete.
\end{proof}

\noindent
{\bf Acknowledgments}

This work is funded by Vietnam National Foundation for Science and Technology Development (NAFOSTED) under Grant No. 101.04-2016.18.

\end{document}